\documentclass[11pt]{amsart}

\usepackage{amsmath}
\usepackage{fullpage}
\usepackage{xspace}
\usepackage[psamsfonts]{amssymb}
\usepackage[latin1]{inputenc}
\usepackage{graphicx,color}
\usepackage[curve]{xypic} 
\usepackage{hyperref}
\usepackage{graphicx}

\usepackage{amsmath}%
\usepackage{amsthm}%
\usepackage{amscd}
\usepackage{amsfonts}%
\usepackage{amssymb}%
\usepackage{graphicx}

\usepackage{mathrsfs}

\usepackage{tikz}
\usetikzlibrary{matrix,arrows}

\usepackage{tikz-cd}

\newtheorem{theorem}{Theorem}[section]

\newtheorem{conjecture}[theorem]{Conjecture}

\newtheorem{corollary}[theorem]{Corollary}

\newtheorem{lemma}[theorem]{Lemma}

\theoremstyle{remark}

\numberwithin{equation}{section}

\newcommand{\pfrak}{\mathfrak{p}}
\newcommand{\Pfrak}{\mathfrak{P}}

\newcommand{\Hcal}{\mathscr{H}}

\newcommand{\Jcal}{\mathscr{J}}

\newcommand{\Z}{\mathbb{Z}}
\newcommand{\C}{\mathbb{C}}

\newcommand{\Q}{\mathbb{Q}}
\newcommand{\R}{\mathbb{R}}

\newcommand{\rad}{\mathrm{rad}}

\newcommand{\Norm}{\mathrm{Norm}}

\newcommand{\Gal}{\mathrm{Gal}}

\newcommand{\Nm}{\mathrm{Nm}}

\newcommand{\hbf}{\mathbf{h}}

\makeatletter
\@namedef{subjclassname@2020}{%
  \textup{2020} Mathematics Subject Classification}
\makeatother

  \DeclareFontFamily{U}{wncy}{}
    \DeclareFontShape{U}{wncy}{m}{n}{<->wncyr10}{}
    \DeclareSymbolFont{mcy}{U}{wncy}{m}{n}
    \DeclareMathSymbol{\Sha}{\mathord}{mcy}{"58}

\begin{document}
\title[]{Power-saving bounds for Thue--Mahler and Mordell equations}
\author{Hector Pasten}
\address{ Departamento de Matem\'aticas,
Pontificia Universidad Cat\'olica de Chile.
Facultad de Matem\'aticas,
4860 Av.\ Vicu\~na Mackenna,
Macul, RM, Chile}
\email[H. Pasten]{hector.pasten@uc.cl}%

\thanks{H.P. was supported by ANID Fondecyt Regular grant 1230507 from Chile.}
\date{\today}
\subjclass[2020]{Primary: 11J25; Secondary: 11J86, 11D25} %
\keywords{Thue--Mahler equations, Mordell's equation, Szpiro's conjecture, linear forms in logarithms, regulators}%

\begin{abstract} We prove a new effective bound for cubic Thue--Mahler equations with power-saving dependence on the regulator. This has some applications. First, we improve Stark's bound $\log \max\{|x|,|y|\} \ll_\epsilon |k|^{1+\epsilon}$ for the integer solutions of Mordell's equation $y^2=x^3+k$ ($k$ a non-zero integer) by reducing the exponent $1+\epsilon$ to $1/2+\epsilon$; this is the first power-saving improvement without restrictions on $k$ in more than 50 years. Secondly, for integer squares and cubes of size $\asymp T$ we improve the known unconditional separation lower bound $(\log T)^{1-o(1)}$ obtained by Stark in 1973 to $(\log T)^{2-o(1)}$. Finally, we obtain a power-saving improvement in the conductor aspect of the strongest currently available bounds for Frey's height conjecture (a strengthening of Szpiro's conjecture) in the case of elliptic curves over $\mathbb{Q}$ with integral $j$-invariant.
\end{abstract}

\maketitle



\section{Introduction}

The goal of this article is to obtain a new bound for the size of the solutions of Thue--Mahler equations with a power-saving improvement in the regulator (and discriminant) aspects. This has applications to Mordell's equation, lower bounds for the difference of a square and a cube,  and  Szpiro's conjecture. Since the result for Thue--Mahler equations is more technical to state, let us begin the discussion with the applications.

\subsection{Mordell's equation} For a non-zero integer $k$, the study of the Diophantine equation 
$$
y^2=x^3+k
$$ 
goes back at least to Bachet in the 17th century. These equations define what is now known as Mordell elliptic curves, after the substantial work of Mordell on this topic, see for instance \cite{MordellTalk}. Among other things, Mordell proved that for any given integer $k\ne 0$, the equation has only a finite number of integer solutions $x$ and $y$. The result, however, was ineffective.

The first effective bound for the size of the solutions in terms of $k$ came from the work of Baker \cite{Baker} in 1968, where he proved the effective estimate
$$
\log \max\{|x|,|y|\} \ll |k|^{10000}
$$
with an effective implicit constant. The dependence on $k$ was later improved by Stark in 1973 when he proved \cite{Stark}
$$
\log \max\{|x|,|y|\} \ll_\epsilon |k|^{1+\epsilon}
$$
with an effective implicit constant depending only on $\epsilon>0$. This was later refined by several authors including Sprindzuk \cite{Sprindzuk}, Juricevic \cite{Juricevic}, and von K\"anel \cite{vonKanel} but, to the best of our knowledge, always in the regime $|k|^{1+o(1)}$ (although von K\"anel's work gives an improvement when $k$ has large exponents in its prime factorization; compare with Theorem \ref{ThmMordellTruncated} below). We prove the first power saving since Stark's bound more than five decades ago, replacing $|k|^{1+\epsilon}$ by $|k|^{1/2+\epsilon}$.
\begin{theorem}[Bound for Mordell's equation]\label{ThmMordell}  Let $k$ be a non-zero integer. The solutions of $y^2=x^3+k$ in $\Z$ satisfy
$$
\log \max\{|x|,|y|\} \ll |k|^{1/2} (\log(2|k|))^4
$$
where the implicit constant is effective and absolute.
\end{theorem}
An immediate consequence of Theorem \ref{ThmMordell} is:
\begin{theorem}[Gap between squares and cubes] For all integers $x,y$ with $x^3\ne y^2$ we have
$$
|x^3-y^2| \gg \frac{ (\log X)^2}{ (\log\log X)^8}, \quad X=\max\{3,|x|,|y|\}
$$
where the implicit constant is effective and absolute.
\end{theorem}
Before our work, all available unconditional lower bounds for $|x^3-y^2|$ in terms of $X$ had the form $(\log X)^{1-o(1)}$, see for instance \cite{Stark}. Thus, our result can be considered as an (admittedly modest) step forward in the direction of Hall's conjecture \cite{Hall, Elkies}.

In fact, we prove a stronger version of Theorem \ref{ThmMordell} which allows truncation in the prime factorization of $k$ as in von K\"anel's work \cite{vonKanel}. For a non-zero integer $n$ we define
$$
\underline{n} = \prod_{p|n} p^{\min\{2, v_p(n)\}}
$$ 
where $p$ varies over (positive) prime divisors of $n$ and $v_p(n)$ is the exponent of $p$ in the factorization of $n$. Since $\underline{n}\le |n|$ we see that Theorem \ref{ThmMordell} is a consequence of the following estimate, which is the main result on Mordell's equation in this article.
\begin{theorem}[Bound for Mordell's equation, with truncation]\label{ThmMordellTruncated}  Let $k$ be a non-zero integer. The solutions of $y^2=x^3+k$ in $\Z$ satisfy
$$
\log \max\{|x|,|y|\} \ll (\underline{k})^{1/2} (\log(2\underline{k}))^2\log(2|k|)\log\left(\underline{k} \log(3|k|)\right)
$$
where the implicit constant is effective and absolute.
\end{theorem}
We note that von K\"anel \cite{vonKanel} had a similar bound, but with a factor $\underline{k}$ rather than $(\underline{k})^{1/2}$. 

Let $\rad(n)=\prod_{p|n}p$ be the radical of a non-zero integer $n$. Since $(\underline{n})^{1/2}\le \rad(n)$ we deduce

\begin{theorem}[Bound for Mordell's equation, with radical]\label{ThmMordellrad}  Let $k$ be a non-zero integer. The solutions of $y^2=x^3+k$ in $\Z$ satisfy
$$
\log \max\{|x|,|y|\} \ll \rad(k) (\log(2\, \rad(k)))^2\log(2|k|)\log\left(\rad(k)\log(3|k|)\right)
$$
where the implicit constant is effective and absolute.
\end{theorem}
\subsection{The height conjecture} Let us discuss an application of Theorem \ref{ThmMordellrad} in the context of elliptic curves and Szpiro's conjecture.

For an elliptic curve $E$ over $\Q$ let $\Delta_E$, $N_E$, and $h(E)$ be the absolute value of its minimal discriminant, its conductor, and its Faltings height relative to $\Q$ as normalized in \cite{Silverman}. One has the standard facts that 
$$
N_E\le \Delta_E \quad\mbox{ and }\quad \log \Delta_E \ll \max\{1,h(E)\},
$$ 
where the implicit constant is effective, see \cite{Silverman}. Szpiro \cite{Szpiro} conjectured that 
$$
\log \Delta_E \le c \log N_E
$$ 
for some absolute constant $c>0$, and it is well-known that this conjecture implies a form of the $abc$ conjecture. Frey \cite{Frey} proposed a stronger version of Szpiro's conjecture:
\begin{conjecture}[The height conjecture] There is an absolute constant $c>0$ such that for all elliptic curves $E$ over $\Q$ we have
$$
h(E) \le c \log N_E.
$$
\end{conjecture}
 This implies Szpiro's conjecture because $\log \Delta_E \ll \max\{1,h(E)\}$. The height conjecture has its advantages, as it is easier to connect with the theory of modular forms, see for instance \cite{MurtyPasten}. But even if one knows Szpiro's conjecture for a family of elliptic curves, it is not known how to deduce the height conjecture for that family. The issue is that $h(E)$ is a global Arakelov-theoretical invariant, $\log \Delta_E$ is its non-Archimedean part, and even under a good bound for $\log\Delta$ one would still need additional control on the Archimedean component of $h(E)$.

The best unconditional result towards the height conjecture available today is due to R. Murty and the author \cite{MurtyPasten}: we have the (effective) bound
\begin{equation}\label{EqnMP}
h(E) \ll N_E \log N_E.
\end{equation}

A notable case where Szpiro's conjecture is established but one does not know how to improve on the previous estimate for $h(E)$ is the case of elliptic curves with integral $j$-invariant. Here, Szpiro's conjecture in a sharp form was proved by Pesenti and Szpiro \cite{PesentiSzpiro}. Theorem \ref{ThmMordellrad} implies the following power-saving improvement on \eqref{EqnMP}:

\begin{theorem}[Bound for the Faltings height]\label{ThmHeightConj} For all elliptic curves $E$ over $\Q$ with integral $j$-invariant we have
$$
h(E) \ll N_E^{1/2} (\log N_E)^4
$$
where the implicit constant is absolute and effective.
\end{theorem}
\subsection{Thue--Mahler equations} Our main result is the following effective bound for cubic Thue--Mahler equations.

\begin{theorem}[Bound for cubic Thue--Mahler equations, version 1]\label{ThmThueMahler1} Let $F\in\Z[U,V]$ be an irreducible binary cubic form with discriminant of absolute value $D$ and let $H$ be the maximal absolute value of the coefficients of $F$. Let $K$ be the number field generated by a root of $F(U,1)$ and let $L$ be the splitting field of $F$. Let $p_1,...,p_r$ be distinct prime numbers (we allow $r=0$ in which case the list is empty), let $S$ be the set of places of $K$ consisting of the Archimedean ones and those above the $p_i$'s, and write $s=\#S$. Let $R_S$ be the $S$-regulator of $K$ and let $P_+$ be the maximal value of $\Norm(\Pfrak)$ as $\Pfrak$ varies over primes of $L$ dividing one of the $p_i$'s if $r>0$, and set $P_+=2$ if $r=0$.

 If $m\ne 0$ is an integer coprime to the $p_i$'s and $w\ne 0$ is an integer supported on the $p_i$'s, then the solutions $x,y\in\Z$ of $F(x,y)=wm$ with $\gcd(x,y)=1$ satisfy
$$
\log \max\{|x|,|y|\} \ll \log H +  (cs )^{2s}\frac{P_+}{\log P_+}R_S \log(2|m|D) \log(2+P_+R_S\log(2|m|D) )
$$
for a certain constant $c>0$, where $c$ and the implicit constant are absolute and effective.
\end{theorem}

As the statement is somewhat technical, let us give a consequence purely over $\Q$, not involving invariants from $K$ and $L$.

\begin{theorem}[Bound for cubic Thue--Mahler equations, version 2]\label{ThmThueMahler2} Let $F\in\Z[U,V]$ be an irreducible binary cubic form with discriminant of absolute value $D$ and let $H$ be the maximal absolute value of the coefficients of $F$. Let $p_1,...,p_r$ be distinct prime numbers ($r=0$ is allowed, in which case the list is empty) and let $Q$ be the largest of them when $r>0$, or $Q=2$ if $r=0$. Let $m\ne 0$ be an integer coprime to the $p_i$'s and $w\ne 0$ an integer supported on the $p_i$'s. Define
$$
\Theta = \prod_{i=1}^r \log p_i \quad \mbox{ (this is $1$ if $r=0$)}.
$$
Then all the solutions $x,y\in\Z$ of $F(x,y)=wm$ with $\gcd(x,y)=1$ satisfy
$$
\begin{aligned}
\log& \max\{|x|,|y|\} \\
&\ll \log H +  (c\cdot (r+1))^{6r}\Theta^3\frac{Q^3}{\log Q}(\underline{D})^{1/2}(\log(2\underline{D}))^2\log(2|m|D) \log(Q\underline{D}\Theta\log (3|m|D))
\end{aligned}
$$
for a certain constant $c>0$, where $c$ and the implicit constant are absolute and effective.
\end{theorem}
As it might be useful in some applications, let us give a weaker consequence with a bound of a much simpler shape.
\begin{corollary}[A simpler Thue--Mahler bound]\label{CoroThueMahler} Let $F\in\Z[U,V]$ be an irreducible binary cubic form with coefficients of absolute value at most $H$ and let $D$ be the absolute value of its discriminant. For all $x,y$ coprime integers and for every $\epsilon>0$, we have
$$
\log\max\{|x|,|y|\}\ll_\epsilon \log H + e^{(6+\epsilon)Q(x,y)}(\underline{D})^{1/2}(\log(2\underline{D}))^2(\log(2D))^2
$$
where $Q(x,y)$ is the largest prime factor of $F(x,y)$ (or defined as $2$ if $F(x,y)=\pm 1$), and the implicit constant is effective and only depends on $\epsilon$. In particular,
$$
\log\max\{|x|,|y|\}\ll_\epsilon \log H + e^{(6+\epsilon)Q(x,y)}D^{1/2}(\log(2D))^4.
$$
\end{corollary}

The key aspect of these bounds is the dependence on $R_S$ and $\underline{D}$. Previous works (see for instance  Theorem 3 in \cite{BugeaudGyory}) had a term $R_S^2$, as opposed to our bound in Theorem \ref{ThmThueMahler1}. Upon using classical regulator-discriminant bounds of Landau \cite{Landau}, one gets a term $(\underline{D})^{1+o(1)}$ rather than our $(\underline{D})^{1/2+o(1)}$ in Theorem \ref{ThmThueMahler2} and Corollary \ref{CoroThueMahler}.

\subsection{About the proofs} That Theorem \ref{ThmMordell} follows from Theorem \ref{ThmThueMahler2} is completely classical and goes back to the work of Baker \cite{Baker} and Stark \cite{Stark} (actually, we go from Theorem \ref{ThmThueMahler2} to Theorem \ref{ThmMordellTruncated}). For the convenience of the reader we recall the argument in Section \ref{SecThueToMordell}, combining the classical expositions by Baker \cite{Baker} and Stark \cite{Stark} with the modern presentation of Bennett--Ghadermarzi \cite{BennettGhadermarzi}.

The approximation argument via linear forms in logarithms for the Thue equation is, for the most part, classical. There is however a crucial difference that allows us to lower Stark's exponent $1+\epsilon$ to $1/2+\epsilon$, as we now explain.

In the classical analysis of the cubic Thue equation (see for instance \cite{Juricevic} for a particularly clean exposition) the approximating element takes the form $\eta=u\cdot \alpha$ where $u$ comes from the unit group of a certain number field $K$, and $\alpha$ is an additional algebraic element. Then $u$ and $\alpha$ are controlled \emph{separately} using geometry of numbers on generators for $O_K^\times$, which contributes with $R^2$ where $R$ is the regulator of $K$. In our setting, however, we use ideas of Okazaki \cite{Okazaki} to introduce the group $O^\times_K\langle \omega\rangle$ for a suitable $\omega$, and then we apply geometry of numbers \emph{once} on this structure (cf. Lemma \ref{LemmaReg}), so we only get $R$ once. Actually, we work with a group of the form $O^\times_{K,S}\langle \omega\rangle$ to include the Mahler aspect (primes with unspecified  powers) of the Thue--Mahler equations. 

Okazaki \cite{Okazaki}  focuses on bounding the \emph{number} of solutions of $F(U,V)=1$ for a homogeneous cubic $F$ but, as outlined above, those ideas are very efficient for bounding the \emph{size} of the solutions too. 

The use of AI in this work is described in the Acknowledgments section.

\subsection{Notation} From now on, all implicit constants in Vinogradov's notation $\ll$  are effective and any dependence on parameters will be indicated by subscripts. If $f$ is an integer polynomial, the maximal absolute value of its coefficients is denoted by $H(f)$. If $K$ is a number field and $v$ is a place of it, the associated norm is $\|-\|_v$. These norms are powers of the corresponding absolute values with exponents equal to the local degrees (not normalized to $\Q$) and they satisfy the product formula $\prod_v\|x\|_v=1$ for $x\in K^\times$. For an algebraic number $\alpha\in K$ its (logarithmic) height is 
$$
h(\alpha)=\frac{1}{[K:\Q]} \sum_v \log \max\{1,\|\alpha\|_v\}.
$$
At some point we will use the $L^1$-norm $\|-\|_1$ on real vector spaces. This should not be confused with the $\|-\|_v$ introduced above.

\section{Preliminaries}

\subsection{Number field bounds}

From (1.8.4) in \cite{EvertseGyory} we have the next bound, which is useful when one applies Theorem \ref{ThmThueMahler1}.
\begin{lemma}[$S$-regulator bound]\label{LemmaRS} Let $K$ be a number field, $S$ a finite set of places containing the Archimedean ones, and $S_0$ the set of non-Archimedean places of $S$. Let $\hbf$ and $R$ be the class number and regulator of $K$. Then the $S$-regulator $R_S$ is bounded as follows:
$$
R_S \le \hbf R\prod_{\pfrak\in S_0}\log \Norm(\pfrak).
$$
\end{lemma}
We remark that when $S_0=\emptyset$ the previous bound gives $R_S\le \hbf R$ while we actually have $R_S=R$. So, in general it is better to keep $R_S$ in our bounds.

The next result, due to Landau, is classical and can be found in \cite{Landau}.

\begin{lemma}[Landau]\label{LemmaLandau}
Let $K$ be a degree $d>1$ number field with absolute value of discriminant $D_K$, class number $\hbf$,  and regulator $R$. Then
$$
1\ll_d \hbf R \ll_d D_K^{1/2}(\log D_K)^{d-1}.
$$
\end{lemma}
We remark that a sharper explicit bound can be found in (1.5.2) of \cite{EvertseGyory}.

The next result is the reason why $\underline{D}$ appears in Theorem \ref{ThmThueMahler2}.
\begin{lemma}\label{LemmaHensel} Let $K$ be a cubic number field with absolute value of discriminant $D_K$ and let $D$ be a non-zero integer with $D_K | D$. Then
$$
D_K \ll \underline{D}.
$$
\end{lemma}
\begin{proof} Let $K'$ be a degree $n$ number field and $p$ a prime dividing $D_{K'}$. Mahler's discriminant-exponent bound (see p.58 in \cite{Serre} for the local statement) gives 
$$
v_p(D_{K'}) \le n-1 + n\left\lfloor\frac{\log n}{\log p}\right\rfloor,
$$
see for instance Corollary 2.5 in \cite{MantillaSoler}. Applied to $K'=K$ and $n=3$ we get the result.
\end{proof}

\subsection{Linear forms in logarithms} We will need an input from the theory of linear forms in logarithms. The following result is a special case of Theorem 4.2.1 in Evertse and Gy\"ory \cite{EvertseGyory} with  the choice $\alpha=1$. It follows in fact from the theory of linear forms in (complex and $p$-adic) logarithms in combination with a strong new result from \cite{EvertseGyory} in geometry of numbers. To state it, for a number field $L$ and a place $v$ of it, we let $\Nm(v)=\Norm(\pfrak)$ if $v$ is associated to a prime ideal $\pfrak$, and $\Nm(v)=2$ if $v$ is Archimedean.

\begin{theorem}[Approximation in multiplicative groups]\label{ThmLFL} Let $L$ be a number field of degree $\delta$ and let $v$ be a place of it. Let $\Gamma$ be a finitely generated multiplicative subgroup of $L^\times$ and let $\gamma_1,...,\gamma_s\in \Gamma$ be a system of generators for $\Gamma/\Gamma_{\rm tor}$, with $s\ge 1$ (here, $\Gamma_{\rm tor}$ is the torsion part of $\Gamma$). There is an effective constant $c(\delta)>0$ depending only on $\delta$ such that for every $\gamma\in\Gamma$ different from $1$ we have
$$
-\log\|1-\gamma\|_v \le c(\delta)^{s} \frac{\Nm(v)}{\log \Nm(v)} \left(\log \max\{2,h(\gamma)\}\right)\prod_{i=1}^s \max\{1,h(\gamma_i)\}.
$$
\end{theorem}
Note that Theorem 4.2.1 in \cite{EvertseGyory} uses $\prod_{i=1}^s h(\gamma_i)$ rather than $\prod_{i=1}^s \max\{1,h(\gamma_i)\}$. This weaker version is better suited for our application.

Also, Theorem 4.2.1 has the factor $\log\max\{e,(\Nm(v)h(\gamma))\}$ instead of $\log \max\{2,h(\gamma)\}$ but this is not a problem: our version follows when $h(\gamma)\ge \Nm(v)/\log \Nm(v)$, and when $h(\gamma)<\Nm(v)/\log \Nm(v)$ one has the effective estimates
$$
-\log\|1-\gamma\|_v \ll_\delta h(\gamma)+1 \ll \frac{\Nm(v)}{\log \Nm(v)}.
$$
\subsection{Small generators} The following technical lemma will allow us to control the bound coming from linear forms in logarithms. Here, for a number field $K$ we let $\mu_K$ denote the group of roots of unity of $K$.

\begin{lemma}[Existence of small generators]\label{LemmaReg} Let $K$ be a number field of degree $d$. Let $S$ be a finite set of places of $K$ containing the Archimedean ones, let $s=\#S$, and let $R_S$ be the $S$-regulator of $K$. Assume that for every place $v$ of $\Q$, $S$ contains all or none of the places of $K$ above $v$. Let $\omega\in O_K$ and define $q=|\Norm (\omega)|$ as well as $q'$ the coprime-to-$S$ part of $q$ (i.e., removing from $q$ the prime factors below $S$). Assume that $q'>1$. The group $O_{K,S}^{\times}\langle \omega\rangle$ admits multiplicatively independent elements $\xi_1,...,\xi_s$  whose images modulo torsion generate a subgroup $\Xi \le O_{K,S}^{\times}\langle \omega\rangle/\mu_K$ of index at most $s!$ satisfying the estimate
$$
\prod_{i=1}^s h(\xi_i) \le  \frac{(s+1)!}{(2d)^s}R_S \log q'.
$$
\end{lemma}
\begin{proof} Let $T$ be the set of places of $K$ consisting of $S$ and the primes dividing $\omega$, so that $O_{K,S}^{\times}\langle \omega\rangle\le O_{K,T}^{\times}$. We will use the logarithmic map $\lambda:  O_{K,T}^{\times} \to \R^{T}$ given by
$$
\lambda(\alpha) = (\log\|\alpha\|_v)_{v\in T}.
$$

Let $\tau_1,...,\tau_{s-1}$ be a system of  fundamental $S$-units of $K$, then $\tau_1,...,\tau_{s-1},\omega$ give a basis of $O_{K,S}^{\times}\langle \omega\rangle/\mu_K$ because $q'>1$. Endow $\bigwedge^{s} \R^T$ with the $\|-\|_1$-norm. We claim that
\begin{equation}\label{EqnExterior}
\|\lambda(\tau_1)\wedge\cdots\wedge\lambda(\tau_{s-1})\wedge\lambda(\omega)\|_1 = (s+1) R_S \log q'
\end{equation}
where $R_S$ is the $S$-regulator of $K$.

Indeed, consider the matrix with columns $\lambda(\tau_1),...,\lambda(\tau_{s-1}),\lambda(\omega)$. There are $s$ rows coming from places in $S$, and every row coming from $T'=T-S$ has $0$'s in the first $s-1$ entries (note that $T'\ne\emptyset$ because $q'>1$). Thus, a non-vanishing $s$-minor has either no row from $T'$ or exactly one. 

By the product formula, in the first case the minor has absolute value
$$
R_S\left|\sum_{v\in S} \log \|\omega\|_v\right| = R_S\left|\sum_{v\in T'} -\log \|\omega\|_v\right| =R_S\log q'
$$  
where we have used our hypothesis on $S$. 

In the second case a minor with $T'$-row indexed by a place $\pfrak$ has  absolute value $R_S|\log \|\omega\|_\pfrak |$. There are $s$ such cases for each $\pfrak\in T'$ and their absolute values add up to $sR_S|\log \|\omega\|_\pfrak |$. Adding over $\pfrak\in T'$ gives $sR_S\log q'$. 

Putting the first and the second cases together gives the claimed equality \eqref{EqnExterior}.

The result now follows directly from work of Akhtari--Vaaler, namely, Theorem 1.2 in \cite{AkhtariVaaler} (see the remark after that theorem to convert $L^1$-norms into heights).
\end{proof}


\section{Thue--Mahler equations}
The goal of this section is to prove Theorem \ref{ThmThueMahler1}; at the end we deduce Theorem \ref{ThmThueMahler2}. We keep the same notation as in the statement of Theorem \ref{ThmThueMahler1}.
\subsection{Setup}

Multiplying $F$ by $-1$ and replacing $(w,m)$ by $(|w|,|m|)$ if needed, we can assume that $w, m>0$. We recall that we are under the assumptions that $\gcd(x,y)=1$, that no prime $p_j$ divides $m$, and that $F$ is irreducible. 
 
 Let $L\subseteq \C$ be the splitting field of $F(U,V)$ and let $\rho_1,\rho_2,\rho_3\in L$ and $a\in\Z$ be such that 
$$
F(U,V)=a\prod_{i=1}^3(U-\rho_iV).
$$
We note that $a\rho_j$ are algebraic integers.

As is classical in the study of Thue equations, define the quantities
$$
\beta_i = x-\rho_iy
$$
for $i=1,2,3$. Choose $x',y'\in \Z$ such that $xy'-x'y=1$ and let
$$
G(U,V)=F(xU+x'V, yU + y'V).
$$
The point of this change of variables is that now we get $wm$ as the coefficient of $U^3$:
$$
G(U,V) = F(x,y)U^3 + b_1U^2V+ b_2 UV^2 + b_3 V^3 = wmU^3 + b_1U^2V+ b_2 UV^2 + b_3 V^3
$$
for suitable integers $b_j$. (This change of variables for Thue's equation goes back at least to the work of Bombieri--Schmidt \cite{BombieriSchmidt}.) 

Let 
$$
P(Y)=(wm)^{2}G(Y/(wm),1) = Y^3 + b_1Y^2 + wmb_2Y+(wm)^2b_3.
$$
Then $|{\rm disc}(P)| = (wm)^2|{\rm disc}(F)|=(wm)^2D$. 

Let $\theta_1,\theta_2,\theta_3$ be the roots of $G(U,1)$ labeled such that
$$
(x-y\rho_i)(U-\theta_iV) = (xU+x'V) - \rho_i( yU + y'V).
$$ 
We note that $wm\theta_i$ are the roots of $P$ and they are algebraic integers.

From the previous formula we get
$$
\theta_i = -\frac{x'-y'\rho_i}{x-y\rho_i} = -\frac{x'-y'\rho_i}{\beta_i} 
$$
and we deduce
\begin{equation}\label{EqnDifTheta}
\begin{aligned}
\theta_i-\theta_j &= \frac{x'-y'\rho_j}{\beta_j}- \frac{x'-y'\rho_i}{\beta_i}   = \frac{x'(\beta_i-\beta_j) - y'(\beta_i\rho_j - \beta_j\rho_i)}{\beta_i\beta_j} \\
& = \frac{x'y(\rho_j-\rho_i) - y' x(\rho_j-\rho_i)}{\beta_i\beta_j}  = \frac{\rho_i-\rho_j}{\beta_i\beta_j}. 
\end{aligned}
\end{equation}
Let $K=\Q(\theta_1)\subseteq L\subseteq \C$ and define
$$
\omega  = (wm)^2(\theta_1-\theta_2)(\theta_1-\theta_3)= P'(wm\theta_1)\in K
$$
which is a non-zero algebraic integer because $wm\theta_i$ are distinct algebraic integers. Thus
$$
\omega\in O_K.
$$ 
Note that
\begin{equation}\label{EqnNomega}
|\Norm(\omega)| = |P'(wm\theta_1)P'(wm\theta_2)P'(wm\theta_3)| = |{\rm disc}(P)|=(wm)^2D.
\end{equation}
It is important to observe that $D_K$, the absolute value of the discriminant of $K=\Q(\rho_1)=\Q(\theta_1)$, divides $D=|{\rm disc}(F)|=|{\rm disc}(G)|$, that is,
\begin{equation}\label{EqnDdiv}
D_K|D,\quad\mbox{ thus }D\ge 2.
\end{equation}
This is because of the theory of orders attached to binary forms \cite{BirchMerriman, Nakagawa} as described, for instance, in Section 2.1 of \cite{Wood}.

Consider the field embeddings $\sigma_i:K\to L$ given by $\sigma_i(\theta_1)=\theta_i$ and define the multiplicative morphism
$$
\zeta:K^{\times}\to L^\times,\quad \zeta(\xi) = \frac{\sigma_2(\xi)}{\sigma_3(\xi)}. 
$$
With this at hand, we finally introduce the key quantity in this argument:
\begin{equation}\label{Eqneta}
\eta = \zeta(\omega ) = \frac{P'(wm\theta_2)}{P'(wm\theta_3)} = -\frac{\theta_2-\theta_1}{\theta_3-\theta_1}=-\frac{\rho_1-\rho_2}{\rho_1-\rho_3}\cdot \frac{\beta_3}{\beta_2}.
\end{equation}
This quantity is defined in such a way that
\begin{equation}\label{Eqn1+eta}
1+\eta = 1 -\frac{\theta_2-\theta_1}{\theta_3-\theta_1} = \frac{\theta_3-\theta_2}{\theta_3-\theta_1} = \frac{\rho_3-\rho_2}{\rho_3-\rho_1}\cdot \frac{\beta_1}{\beta_2}.
\end{equation}
In fact, this equation makes it transparent that $\eta$ is not a new quantity: it is obtained from the classical \emph{Siegel  identity}:
$$
(\rho_2-\rho_3)\beta_1+(\rho_3-\rho_1)\beta_2+(\rho_1-\rho_2)\beta_3=0.
$$

We need to get rid of a problematic case before we can go further with the argument.
\begin{lemma}\label{LemmaRoot1} We have that $1+\eta\ne 0$. Furthermore, if $\eta$ is a root of unity, then it is a primitive cubic root of unity and for each place $v$ of $L$ we have $\|1+\eta\|_v= 1$.
\end{lemma}
\begin{proof} That $1+\eta\ne 0$ follows from \eqref{Eqn1+eta} and the fact that the $\theta_j$'s are distinct.

Assume now that $\eta$ is a root of unity. Since $F$ is irreducible, there is a $3$-cycle $\tau\in\Gal(L/\Q)$. From \eqref{Eqneta} and \eqref{Eqn1+eta} we deduce that both $\eta$ and $1+\eta$ are roots of unity, and with our chosen complex embedding $L\subseteq \C$ we get $|\eta|=|1+\eta|=1$. It follows that $\eta=\exp(\pm 2\pi i/3)$.
\end{proof}

From now on, the fact that $1+\eta\ne 0$ will be used without further mention.

We define $X=\max\{3,|x|,|y|\}$. Bounding the size and height of the roots of an integer polynomial by the size of its coefficients gives
\begin{equation}\label{Eqnhrho}
|\rho_i|\ll H\quad\mbox{ and }\quad h(\rho_i)\ll \log(2H).
\end{equation}

We will later need a bound for $h(\eta)$. From \eqref{Eqneta} we get
$$
\eta = - \frac{\rho_2-\rho_1}{\rho_3-\rho_1} \cdot \frac{\beta_3}{\beta_2}
$$
and thus $\eta$ is a fixed rational function of the $\rho_i$ and $x,y$. From \eqref{Eqnhrho} we deduce
\begin{equation}\label{Eqnbdheta}
h(\eta) \ll 1+ \log (HX).
\end{equation}

At this point we consider any labeling of the roots $\rho_i$, but later in some arguments we will need to make a choice of the labeling that depends on a chosen place $v$ of $L$. The main argument will eventually use the Archimedean place induced by  $L\subseteq \C$, but the \emph{proof} of Lemma \ref{Lemmabdz} chooses a non-Archimedean place; this is not a problem because the statement of Lemma \ref{Lemmabdz} does not involve the roots $\rho_i$. So, a more refined notation should be $\eta_v$ for a choice of place $v$ of $L$, but since we will only use one place $v$ outside the proof of Lemma \ref{Lemmabdz}, we consider this to be an unnecessary complication.

\subsection{Outline of the argument}

Now that the main objects have been defined, let us briefly describe the strategy in the special case that $S$ only consists of the Archimedean places. 

In a nutshell, the classical idea going back to Thue is that a solution $x,y$ of $F(x,y)=m$ with large $x$ and $y$ would produce an $x/y$ that approximates a root, say $\rho_1$, too well. Then one confronts this with results of Diophantine approximation. To get effectivity one replaces Diophantine approximation by transcendence theory (linear forms in logarithms), but then the approximating quantity will no longer be 
$$
\frac{x}{y}-\rho_1 = \frac{\beta_1}{y}
$$
but instead
$$
1+\eta = \frac{\rho_3-\rho_2}{\rho_3-\rho_1}\cdot \frac{\beta_1}{\beta_2}.
$$
Still, the term that makes the approximation good is $\beta_1$ (assuming it is small). But the definition of $\eta$ is made in such a way that a suitable power of it belongs to a multiplicative group with controlled generators (Lemma \ref{LemmaReg}), and then one can use Theorem \ref{ThmLFL}. More precisely, we will use that $\zeta$ is a morphism of multiplicative groups and $\omega\in O_K^\times\langle \omega\rangle$, so that
$$
\eta = \zeta(\omega) \in \zeta(O_K^\times\langle \omega\rangle)
$$
(in the case of general $S$ we will use $O_{K,S}^\times\langle \omega\rangle$). The unit part of $O_K^\times\langle \omega\rangle$ might look superfluous for this, but the point is that $O_K^\times$ gives additional freedom to choose small generators (Lemma \ref{LemmaReg}).

It should be stressed that our $\eta$ is quite classical in the literature on Thue equations and it is motivated by Siegel's identity, see for instance Equation (10) in \cite{BiluHanrot} for an explicit mention. Also, the important idea of using a conjugate ratio coming from a group of the form $O_K^\times\langle \omega \rangle$ already appears in \cite{Okazaki}.

Finally, we mention that in order to include $p$-adic places (Mahler's setup in the Thue--Mahler equation) one needs to control the exponent of primes in $w$ using linear forms in $p$-adic logarithms and then revisiting the previous strategy. This step is the content of Section \ref{Secpadic}, following ideas of Coates \cite{Coates1, Coates2}.

\subsection{Lower bound: linear forms in logarithms}\label{SecLower}

We need a preliminary lower bound for $\|1+\eta\|_v$ both in the Archimedean and non-Archimedean cases. The non-Archimedean one will be used in Section \ref{Secpadic} to control the exponents of the factorization of $w$, while the Archimedean one will be used in Section \ref{SecConfront} to confront the main output of Section \ref{SecUpper}, to the effect that $-\eta$ is a good Archimedean approximation of $1$.

From now on, $S$ is the set of Archimedean places of $K$ and the places of $K$ above some of the primes $p_1,...,p_r$, and $R_S$ is the $S$-regulator of $K$.

\begin{lemma}\label{LemmaLowerBound} For each place $v$ of $L$ we have 
$$
-\log\|1+\eta\|_v \le (\kappa s)^{2s}\frac{\Nm(v)}{\log \Nm(v)} \log(2+\log(XH))R_S\log(m^2D)
$$
where $\kappa>0$ is an absolute and effective constant (in particular, independent of $v$).
\end{lemma}
We remark that in this bound $v$ comes from $L$, while $R_S$ comes from $K$.

\begin{proof} When $\eta$ is a root of unity in $L$ we conclude by Lemma \ref{LemmaRoot1}. So we assume from now on that $\eta$ is not a root of unity.

We note that $\omega \in O_K$ as $P$ is monic with integral coefficients. Furthermore, \eqref{EqnNomega} says $q=|\Norm(\omega)|  =w^2m^2D$ and we take $q'$ as the part of $q$ coprime to all $p_i$'s; thus, $m^2|q'$ and $q'|m^2D$.

If $q'=1$ then $\omega\in O_{K,S}^\times$ and it is classical that this group has generators of controlled height. For instance, Proposition 4.3.9 in \cite{EvertseGyory} gives a system of fundamental $S$-units $\xi_1,...,\xi_{s-1}$ with
$$
\prod_{i=1}^{s-1} h(\xi_i) \le s^{2s} R_S
$$
and the reader can check that this is enough to apply Theorem \ref{ThmLFL} in a way similar to (but simpler than) the case $q'>1$ to be addressed below.

So, let us focus on the case $q'>1$. Then we can apply Lemma \ref{LemmaReg}. We write $\kappa_i>0$ for the effective absolute constants that will appear in the argument.

For  $s=\#S$ we obtain multiplicatively independent elements $\xi_1,...,\xi_s \in O_{K,S}^\times\langle \omega\rangle$ whose images modulo torsion generate a multiplicative group $\Xi$ of index $N\le s!$  in $O_{K,S}^\times\langle \omega\rangle/\mu_K$ and which satisfy
$$
\prod_{i=1}^s h(\xi_i) \le (\kappa_1 s)^sR_S \log(mD)
$$
because $\log q'\le 2\log(mD)$. If $\tilde{\Xi}$ is the lifting of $\Xi$ to $O_{K,S}^\times\langle \omega\rangle$, we note that $\omega^N\in \tilde{\Xi}$.

As the $\xi_i$ are not roots of unity (they are multiplicatively independent), this bound gives
\begin{equation}\label{EqnBdProd}
\prod_{i=1}^s \max\{1,h(\xi_i)\} \le (\kappa_2 s)^s R_S \log(mD)
\end{equation}
Here we used the fact that the height of algebraic numbers that are not roots of unity is bounded away from $0$ just in terms of the degree, see for instance \cite{Voutier} and the references therein.

We are concerned with $\|1+\eta\|_v$ but the quantity that will naturally appear is $\|1-(-\eta)^N\|_v$. We can certainly assume $\|1+\eta\|_v \le 1/2$. In that case $\|\eta\|_v < 4$ and considering the factorization of $1-(-\eta)^N$ as difference of powers, we get
$$
\|1-(-\eta)^N\|_v\le N\kappa_3^N\|1+\eta\|_v\le \exp(\kappa_4s^s)\|1+\eta\|_v
$$
and we see that it is enough to prove the claimed lower bound for $\|1-(-\eta)^N\|_v$.

We apply Theorem \ref{ThmLFL} with $\gamma = (-\eta)^N$ and $\gamma_i=\zeta(\xi_i)$ to get 
$$
-\log \|1-(-\eta)^N\|_v\le \kappa_5^{s}\cdot \frac{\Nm(v)}{\log \Nm(v)} \left(\log \max\{2,h(\eta)\}\right)\prod_{i=1}^s \max\{1,h(\gamma_i)\}
$$
where we have used that $h((-\eta)^N)\le s^sh(\eta)$.

We note that for $\xi\in K^\times$ we have
$$
h(\zeta(\xi))=h(\sigma_2(\xi)/\sigma_3(\xi)) \le h(\sigma_2(\xi))+ h(\sigma_3(\xi))  = 2h(\xi).
$$
Hence, $h(\gamma_i) \le 2 h(\xi_i)$ and we get
$$
-\log \|1-(-\eta)^N\|_v\le \kappa_6^{s}\cdot \frac{\Nm(v)}{\log \Nm(v)} \left(\log \max\{2,h(\eta)\}\right)\prod_{i=1}^s \max\{1,h(\xi_i)\}
$$
and the result follows from \eqref{Eqnbdheta} and \eqref{EqnBdProd}.
\end{proof}
\subsection{Controlling $p$-adic contributions}\label{Secpadic}

\begin{lemma}[Controlling the exponents of $w$]\label{Lemmabdz} Let  $p$ be a prime in the list $p_1,...,p_r$, and let  $\Pfrak$ be a prime in $O_L$ above $p$. Then
$$
v_p(w)\log p \ll (\kappa s)^{2s}\frac{\Norm(\Pfrak)}{\log \Norm(\Pfrak)} \log(2+\log(XH))R_S\log(m^2D) +v_p(a)\log p
$$
where $\kappa>0$ is an effective absolute constant.
\end{lemma}
\begin{proof} We can assume that $z=v_p(w)$ is non-zero.

Recall that $a$ is the coefficient of $U^3$ in $F(U,V)$. Let us distinguish two cases.

First we assume that $v_p(a) <  v_p(y)$.  Recall that $x$ and $y$ are coprime and note that
$$
wm = F(x,y)=ax^3 + yf \mbox{ for some }f\in \Z.
$$ 
In this case $p\nmid x$ (because $p|y$) and $z=v_p(w) = v_p(a)\ll (\log H)/\log p$, so the claimed bound holds. We no longer consider this case.

From now on we assume that $v_p(a) \ge v_p(y)$. Note that 
$$
\|a^2w\|_\Pfrak= \|a^2wm\|_\Pfrak=\|a^2F(x,y)\|_\Pfrak=\prod_{j=1}^3 \|a\beta_j\|_\Pfrak
$$
here in the last product each factor is bounded by $1$, as $0\ne a\beta_j\in O_L$. Let us assume that $\|\beta_1\|_\Pfrak$ is minimal (this is achieved by relabeling). Then
$$
\|a\beta_1\|_\Pfrak\le \|a^2w\|_\Pfrak^{1/3}.
$$

Note that $a\beta_i-a\beta_j = y(a\rho_j-a\rho_i)$. Since 
$$
a^2y^6{\rm disc}(F)=y^6\prod_{i<j}(a\rho_j-a\rho_i) ^2 = \prod_{i<j}(a\beta_i-a\beta_j)^2,
$$
it follows that for each $i<j$ we have
$$
\|y\|_\Pfrak \ge \|ya(\rho_j-\rho_i)\|_\Pfrak= \|a\beta_i-a\beta_j\|_\Pfrak\ge \|a^2y^6D\|_\Pfrak^{1/2}
$$
because $0\ne a(\rho_j-\rho_i)\in O_L$.

We note that by minimality of $\beta_1$ and the strong triangle inequality
$$
 \|a^2y^6D\|_\Pfrak^{1/2}\le \|a\beta_i-a\beta_j\|_\Pfrak \le  \|a\beta_j\|_\Pfrak\quad\mbox{ for }j=2,3.
$$

From \eqref{Eqn1+eta}  we get
$$
\begin{aligned}
\|1+\eta\|_\Pfrak &\le \frac{\|ya(\rho_3-\rho_2)\|_\Pfrak}{\|ya(\rho_3-\rho_1)\|_\Pfrak}\cdot\frac{\|a\beta_1\|_\Pfrak}{\|a\beta_2\|_\Pfrak} \\
&\le  \frac{\|y\|_\Pfrak}{\|a^2y^6D\|_\Pfrak^{1/2}}\cdot\frac{\|a^2w\|_\Pfrak^{1/3}}{\|a^2y^6D\|_\Pfrak^{1/2}} \\
&= \frac{\|w\|_\Pfrak^{1/3}}{\|a\|_\Pfrak^{4/3}\|y^5D\|_\Pfrak}
\end{aligned}
$$
and from $v_p(a)\ge v_p(y)$ we get
$$
\|1+\eta\|_\Pfrak^3 \le  \frac{\|w\|_\Pfrak}{\|a\|_\Pfrak^{19}\|D\|_\Pfrak^3}.
$$
We deduce
$$
 z\log p \ll - \log \|1+\eta\|_\Pfrak + v_p(a)\log p+ \log(D).
$$
The result now follows from Lemma \ref{LemmaLowerBound} and absorbing $\log D$.
\end{proof}

\begin{corollary}\label{Corow}
We have 
$$
\log w \ll (c_0 s)^{2s}\frac{P_+}{\log P_+}R_S \log(mD) \log(2+\log (XH)) + \log H,
$$
where $c_0>0$ is an absolute effective constant.
\end{corollary}
\begin{proof} This is a direct computation using Lemma \ref{Lemmabdz}:
$$
\begin{aligned}
\log w &= \sum_{p|w} v_p(w)\log p\\
& \ll s(\kappa s)^{2s}\frac{P_+}{\log P_+}R_S \log(mD) \log(2+\log (XH)) + \sum_{p|w}v_p(a)\log p\\
&\ll (c_0 s)^{2s}\frac{P_+}{\log P_+}R_S \log(mD) \log(2+\log(XH)) + \log H.
\end{aligned}
$$
\end{proof}
\subsection{Upper bound: $1+\eta$ is small}\label{SecUpper}

From now on, $c_1,c_2,...$ are absolute and effectively computable positive constants. We recall that $K\subseteq L \subseteq \C$ and we use the Archimedean place of this embedding, whose absolute value is simply denoted by $|-|$.

The bounds $\prod_{i<j} |\rho_i-\rho_j|^2=D/a^4 \ge H^{-4}$ and $|\rho_i-\rho_j|\ll H$ yield
$$
|\rho_i-\rho_j| \gg H^{-c_1}.
$$
We can and do assume $y\ne 0$.

If $|y|\ll X/H$ for a suitably small implicit constant, then $X = \max\{3,|x|\}$ and $wm=|F(x,y)|=|a\beta_1\beta_2\beta_3|\gg X^3$ so we get
$$
\log X\ll 1+ \log m+\log w
$$
and, using Corollary \ref{Corow} to control the contribution of $w$, one gets
\begin{equation}\label{EqnKeySmall}
\begin{aligned}
\log X&\ll \log m+ s(\kappa s)^{2s}\frac{P_+}{\log P_+} \log(2+\log(XH))R_S\log(m^2D) + \log H\\
&\ll (c_2 s)^{2s}\frac{P_+}{\log P_+} \log(2+\log(XH))R_S\log(mD) + \log H.
\end{aligned}
\end{equation}

For the moment, let us leave the case $|y|\ll X/H$ in standby, keeping in mind that we get the bound \eqref{EqnKeySmall}. Thus, let us now assume
$$
|y|\gg X/H.
$$
We will arrive in \eqref{EqnKeyBig} to a bound of the same shape as \eqref{EqnKeySmall}, and the proof will conclude by analyzing that bound.

From now on we label the roots $\rho_i$'s so that $|\beta_1|$ is minimal.

Minimality of $\beta_1$ gives for $i=2,3$
$$
|y||\rho_i-\rho_1| = |\beta_i-\beta_1| \ll |\beta_i |
$$
so 
$$
|\beta_i|\gg  XH^{-c_3}.
$$
From $wm=|a\beta_1\beta_2\beta_3|$ we get
$$
|\beta_1|\ll wmH^{c_4}X^{-2}.
$$
We deduce
\begin{equation}\label{EqnUpper}
|1+\eta| = \left|  \frac{\rho_3-\rho_2}{\rho_3-\rho_1}\cdot \frac{\beta_1}{\beta_2}  \right| \ll wmH^{c_5} X^{-3}.
\end{equation}

\subsection{Confronting the upper and lower bounds}\label{SecConfront}

From \eqref{EqnUpper} and Lemma \ref{LemmaLowerBound} in the Archimedean case, we get
$$
3\log X - \log (wm)-c_5 \log H \ll  (c_6s)^{2s}\log \max\{2, h(\eta)\}R_S\log(m^2D).
$$
From \eqref{Eqnbdheta} and absorbing $\log m$ we get
$$
\log X \ll \log(wH) + (c_6s)^{2s}R_S\log(mD) \log(2+\log (XH) ).
$$

From this and Corollary \ref{Corow} we deduce
\begin{equation}\label{EqnKeyBig}
\log X \ll (c_7 s)^{2s}\frac{P_+}{\log P_+}R_S \log(mD) \log(2+\log (XH)) +\log H
\end{equation}
where we assume $c_7>c_2$ (see \eqref{EqnKeySmall}). Thus, this bound holds in both cases: $|y| \ll X/H$ and $|y| \gg X/H$.

If $\log X\ll \log(H)$ we are done, so we assume $\log X\gg \log (H)$ with a sufficiently large (absolute and effective) implicit constant. Then

$$
\log X \ll (c_7 s)^{2s}\frac{P_+}{\log P_+}R_S \log(mD) \log(\log X ) 
$$
and we deduce
$$
\log X \ll (c_8 s)^{2s}\frac{P_+}{\log P_+}R_S \log(mD) \log(2+P_+R_S \log(mD) ). 
$$
Adding $\log H$ back to cover the case $\log X\ll \log(H)$, we finally obtain Theorem \ref{ThmThueMahler1}. \qed 

\subsection{From Theorem \ref{ThmThueMahler1} to Theorem \ref{ThmThueMahler2}} Let $\hbf$ and $R$ be the class number and regulator of $K$, and let $S_0$ be the set of non-Archimedean places in $S$. Define
$$
\Omega=\prod_{\pfrak\in S_0}\log \Norm(\pfrak).
$$

From Lemma \ref{LemmaRS} we get $R_S\le\hbf R \Omega$.

The bounds $P_+\le Q^3$, $s\le 3r+3$, and $\Omega \ll 27^r\Theta^3$ follow from considering the degrees of the number fields $K$ and $L$, the fact that $\Gal(L/\Q)$ is a subgroup of $S_3$ containing a $3$-cycle, and splitting of primes (note that the residue degree in $L$ is bounded by $3$).

Finally, to bound $\hbf R$ in terms of $\underline{D}$ we first apply Lemma \ref{LemmaLandau} to get a bound in terms of $D_K$, and then we use Lemma \ref{LemmaHensel} together with the divisibility $D_K|D$ from \eqref{EqnDdiv} in order to get a bound in terms of $\underline{D}$. These reductions show how Theorem \ref{ThmThueMahler2} follows from Theorem \ref{ThmThueMahler1}.

\subsection{From Theorem \ref{ThmThueMahler2} to Corollary \ref{CoroThueMahler}}

Corollary \ref{CoroThueMahler} is obtained from Theorem \ref{ThmThueMahler2} after taking $m=1$, $r$ as the number of prime factors of $F(x,y)$, $Q=Q(x,y)$, and noticing the following estimates:
\begin{itemize}
\item By the prime number theorem $(r+1)^r \ll_\epsilon \exp((1+\epsilon)P_r)\le e^{(1+\epsilon)Q}$ where $P_r$ is the $r$-th prime number.
\item $\Theta \le (\log Q)^{r} \le (\log Q)^{O(Q/\log Q)} \ll_\epsilon e^{\epsilon Q}$ because $r$ is at most the number of primes up to $Q$, which is $O(Q/\log Q)$.
\end{itemize}
Then the lower order factors are absorbed into the $\epsilon$-part of $e^{(6+\epsilon)Q(x,y)}$.


\section{From Thue to Mordell} \label{SecThueToMordell}

In this section we prove Theorem \ref{ThmMordellTruncated}.

The argument to deduce a bound for Mordell's equation from bounds for the Thue equation is classical (see for instance \cite{Stark}) but, for the convenience of the reader, it will be briefly recalled here. 

Take a non-zero integer $k$. Let $x,y\in \Z$ be integral solutions of $y^2=x^3+k$. Let us consider the cubic form in $S$ and $T$
$$
G(S,T) = S^3 - 3xST^2 + 2yT^3.
$$
The absolute value of its discriminant is $D=108|x^3-y^2|=108|k|$. As explained by Stark in Section 4 of \cite{Stark}, Baker \cite{Baker} shows that there is a  change of variables $\gamma\in {\rm GL_2}(\Z)$ that transforms $G(S,T)$ into a homogeneous cubic $F(U,V) = G(\gamma(U,V))$ with $H(F)\le D^{1/2}$.

Let $(u,v) = \gamma^{-1}(1,0)\in\Z^2$, then
$$
F(u,v)= G(1,0)=1.
$$
Thus, $(u,v)$ is a solution of the Thue equation $F(U,V)=1$ with $\gcd(u,v)=1$ and $H=H(F)\le D^{1/2} \ll |k|^{1/2}$. Let us write $X=\max\{3,|u|,|v|\}$. 

If $F$ is reducible then it is an elementary argument to get $\log X\ll \log (2H) \ll \log (2|k|)$, see for instance Section 2(II) in \cite{Baker}. 

If $F$ is irreducible then Theorem \ref{ThmThueMahler2} with $r=0$ and $m=1$, together with the bound $\log (2H) \ll \log (2|k|)$ give

$$
\begin{aligned}
\log X &\ll \log H + (\underline{D})^{1/2}(\log \underline{D})^2(\log D)\log (\underline{D}\log(3D)) \\
&\ll (\underline{k})^{1/2}(\log (2\underline{k}))^2(\log(2|k|))\log(\underline{k}\log (3|k|)).
\end{aligned}
$$

In either case

$$
\log \max\{3,|u|,|v|\} =\log X\ll (\underline{k})^{1/2}(\log (2\underline{k}))^2(\log(2|k|))\log(\underline{k}\log (3|k|)).
$$

It remains to recover $x$ and $y$ from $u$ and $v$. We follow Section 2 of Bennett--Ghadermarzi \cite{BennettGhadermarzi}.

 If $\Hcal_B$ and $\Jcal_B$ denote the Hessian and the Jacobian covariants of a binary cubic form $B$ then for $G(S,T)$ (denoted as $F$ in \cite{BennettGhadermarzi}) one has
 $$
 \begin{aligned}
 |x| &\ll |\Hcal_G(1,0)| = |\Hcal_G(\gamma(u,v))| =  |\Hcal_F(u,v)|\\
 |y| &\ll |\Jcal_G(1,0)| = |\Jcal_G(\gamma(u,v))| =  |\Jcal_F(u,v)|.
 \end{aligned}
 $$
Thus, $|x|$ and $|y|$ are bounded by polynomial expressions in $u$, $v$, and the coefficients of $F$. Thus, $\log \max\{2,|x|,|y|\} \ll \log |k| + \log X$ and Theorem \ref{ThmMordellTruncated} follows. \qed


\section{Bounds for the Faltings height} 

In this section we prove Theorem \ref{ThmHeightConj}.

Let $E$ be an elliptic curve over $\Q$ with integral $j$-invariant. By Pesenti and Szpiro \cite{PesentiSzpiro} we have 
\begin{equation}\label{EqnPS}
\log \Delta_E\ll \log N_E, 
\end{equation}
and since all the primes of bad reduction for $E$ are of additive type (the $j$-invariant is an integer), we have that 
\begin{equation}\label{EqnAdditive}
\rad(\Delta_E) \ll N_E^{1/2}.
\end{equation}
Write a  short Weierstrass equation $y^2 = x^3+Ax+B$ with $A,B$ integers that minimize the absolute value $\Delta'_E$ of the discriminant of the cubic among all such equations. Due to issues at $2$ and $3$ it is not always the case that $\Delta'_E=\Delta_E$, but nevertheless $\Delta'_E| M \Delta_E$ for an effective absolute integer constant $M$ supported at $2$ and $3$; this can be seen from the standard formulas for $\Delta$ using the elliptic curve quantities $c_4$ and $c_6$ for a minimal general Weierstrass equation. One knows \cite{Silverman}:
$$
h(E) \ll \log\max\{2, |A|, |B|\}.
$$
On the other hand,
$$
4A^3+27B^2 = C, \quad \mbox{ with }|C|=\Delta'_E.
$$
Hence
$$
(108B)^2 - (-12 A)^3 = k, \quad k= 432 C. 
$$

From \eqref{EqnAdditive} we deduce $\rad (k) \le 6 \, \rad(\Delta_E) \ll N_E^{1/2}$. Theorem \ref{ThmMordellrad} then gives
$$
\log\max\{2, |A|, |B|\} \ll N_E^{1/2} (\log N_E)^2 \log (|k|)\log(N_E\log(|k|)).
$$

Finally, we note that $\log (|k|) \ll \log \Delta_E \ll \log N_E$ by \eqref{EqnPS}, and Theorem \ref{ThmHeightConj} follows. \qed

\section{Acknowledgments}

This research was supported by ANID Fondecyt Regular grant 1230507 from Chile. We sincerely thank Yuri Bilu and K\'alm\'an Gy\"ory for valuable feedback on an earlier version of this manuscript.

In the preparation of this work, AI models were used to suggest initial exploratory directions, assist with computations, and proofread earlier versions of the manuscript. The author independently reconstructed and verified all the mathematical content that came from the models, and assumes full responsibility for the paper. The AI models used were OpenAI's ChatGPT-5.6 Sol Pro and Anthropic's Claude Opus 5 and Claude Fable 5.



\begin{thebibliography}{9}         


\bibitem{AkhtariVaaler} S. Akhtari, J. Vaaler. \emph{A bound for the exterior product of S-units}. Algebra \& Number Theory,  18(9), (2024), 1589--1617.
 
\bibitem{Baker} A. Baker, \emph{Contributions to the theory of diophantine equations: II. The diophantine equation $y^2=x^3+k$}. Phil. Trans. Royal Soc. (London), A 263 (1968), 193--208.

\bibitem{BennettGhadermarzi} M. Bennett, A. Ghadermarzi, \emph{Mordell's equation: a classical approach}. LMS J. Comput. Math. 18 (2015), no. 1, 633--646.

\bibitem{BiluHanrot} Y. Bilu, G. Hanrot, \emph{Solving Thue equations of high degree}. Journal of Number Theory, 60 (1996), 373--392.

\bibitem{BirchMerriman} B. Birch, J. Merriman, \emph{Finiteness theorems for binary forms with given discriminant}. Proc. London Math. Soc. (3) 24 (1972), 385--394.

\bibitem{BombieriSchmidt} E. Bombieri, W. Schmidt, \emph{On Thue's equation}. Invent. Math. 88 (1987), 69--81.


\bibitem{BugeaudGyory} Y. Bugeaud, K. Gy\"ory, \emph{Bounds for the solutions of Thue-Mahler equations and norm form equations}. Acta Arith. 74 (1996), 273--292.

\bibitem{Coates1} J. Coates, \emph{An effective $p$-adic analogue of a theorem of Thue}. Acta Arith. 15 (1968/69), 279--305.

\bibitem{Coates2} J. Coates, \emph{An effective $p$-adic analogue of a theorem of Thue. II. The greatest prime factor of a binary form}. Acta Arith. 16 (1969/70), 399--412.

\bibitem{Elkies} N. Elkies, \emph{Rational points near curves and small nonzero $|x^3-y^2|$ via lattice reduction}. Algorithmic number theory (Leiden, 2000), 33--63, Lecture Notes in Comput. Sci., 1838, Springer, Berlin, 2000.

\bibitem{EvertseGyory} J.-H. Evertse, K. Gy\"ory, \emph{Unit equations in Diophantine number theory}. Cambridge Studies in Advanced Mathematics, 146. Cambridge University Press, Cambridge, 2015.

\bibitem{Frey} G. Frey, \emph{Links between solutions of A-B=C and elliptic curves}. Number theory (Ulm, 1987), 31--62, Lecture Notes
in Math., 1380, Springer, New York, (1989).

\bibitem{Hall} M. Hall, \emph{The Diophantine equation $x^3-y^2=k$}. Computers in number theory (Proc. Sci. Res. Council Atlas Sympos. No. 2, Oxford, 1969), pp. 173--198, Academic Press, London-New York, 1971.

\bibitem{Juricevic} R. Juricevic, \emph{Explicit estimates of solutions of some Diophantine equations}. Funct. Approx. Comment. Math., 38 (2) (2008) 171--194.

\bibitem{vonKanel} R. von K\"anel, \emph{Integral points on moduli schemes of elliptic curves}. Trans. London Math. Soc. (2014) 1(1) 85--115.

\bibitem{Landau} E. Landau, \emph{Verallgemeinerung eines P\'olyaschen Satzes auf algebraische Zahlk\"orper}. Nachr. Kgl. Ges. Wiss. G\"ottingen, Math.-Phys. Kl. (1918), 478--488.

\bibitem{MantillaSoler} G. Mantilla-Soler, \emph{The spinor genus of the integral trace}. Transactions of the American Mathematical Society, 369(3), (2017), 1611--1626.


\bibitem{MordellTalk} L. Mordell, \emph{A chapter in the theory of numbers}. Cambridge University Press, Cambridge, 1947.

\bibitem{MurtyPasten} M. R. Murty, H. Pasten, \emph{Modular forms and effective Diophantine approximation}. J. Number Theory 133 (2013), no. 11, 3739--3754.


\bibitem{Nakagawa} J. Nakagawa, \emph{Binary forms and orders of algebraic number fields}. Invent. Math. 97(2) (1989), 219--235.

\bibitem{Okazaki} R. Okazaki, \emph{Geometry of a cubic Thue equation}. Publ. Math. Debrecen 61, (2002), 267--314.

\bibitem{PesentiSzpiro} J. Pesenti, L. Szpiro, \emph{In\'egalit\'e du discriminant pour les pinceaux elliptiques \`a r\'eductions quelconques}. Compositio Math. 120 (2000), no. 1, 83--117.

\bibitem{Serre} J.-P. Serre, \emph{Local fields}. Graduate Texts in Mathematics, 67. Springer-Verlag, New York-Berlin, 1979. viii+241 pp.

\bibitem{Silverman} J. Silverman, \emph{Heights and elliptic curves}. Arithmetic geometry (Storrs, Conn., 1984), 253--265, Springer, New York, 1986.


\bibitem{Sprindzuk} V. Sprindzuk, \emph{Classical Diophantine Equations}. Lect. Notes in Math. 1559 (1993), Springer-Verlag, Berlin \& New York.

\bibitem{Stark} H. Stark, \emph{Effective estimates of solutions of some diophantine equations}. Acta Arith. 24 (1973) 251--259.

\bibitem{Szpiro} L. Szpiro, \emph{Pr\'esentation de la th\'eorie d'Arak\'elov}. Current trends in arithmetical algebraic geometry (Arcata, Calif., 1985), 279--293, Contemp. Math., 67, Amer. Math. Soc., Providence, RI, 1987.

\bibitem{Voutier} P. Voutier, \emph{An effective lower bound for the height of algebraic numbers}. Acta Arith. 74 (1996), 81--95.

\bibitem{Wood} M. Wood, \emph{Rings and ideals parameterized by binary $n$-ic forms}. J. Lond. Math. Soc. (2) 83(1) (2011), 208--231.

\end{thebibliography}
\end{document}